\documentclass[11pt]{amsart}

\usepackage{newtxtext,newtxmath}   
\usepackage[final]{microtype}      

\usepackage[margin=3.5cm]{geometry}

\usepackage{amsmath}               
\usepackage{mathtools}
\mathtoolsset{showonlyrefs=true}

\usepackage{enumitem}
\setlist[itemize]{topsep=3pt,itemsep=2pt}
\setlist[enumerate]{topsep=3pt,itemsep=2pt}

\usepackage{graphicx}
\graphicspath{{./Pictures/}}
\DeclareGraphicsExtensions{.pdf,.png,.jpg} 

\usepackage[hidelinks,bookmarksdepth=3]{hyperref}
\usepackage{hypcap} 

\usepackage{caption}

\expandafter\def\expandafter\normalsize\expandafter{%
  \normalsize
  \setlength\abovedisplayskip{7pt}%
  \setlength\belowdisplayskip{7pt}%
  \setlength\abovedisplayshortskip{5pt}%
  \setlength\belowdisplayshortskip{5pt}%
}

\numberwithin{equation}{section}

\theoremstyle{definition}

\newtheorem{theorem}{Theorem}
\newtheorem{lemma}[theorem]{Lemma}
\newtheorem{proposition}[theorem]{Proposition}

\theoremstyle{definition}
\newtheorem{definition}[theorem]{Definition}
\newtheorem{remark}[theorem]{Remark}

\author{Matilde Gianocca}
\address{Department of Mathematics, ETH Z\"urich, Switzerland}
\email{matilde.gianocca@math.ethz.ch}

\title{Note on Energy index and first eigenvalue of minimal surfaces in spheres}
\date{June 2025}

\subjclass[2020]{}

\hypersetup{
  pdftitle={Rigidity in the Ginzburg--Landau approximation of harmonic spheres},
  pdfauthor={Matilde Gianocca}
}
\date{}

\begin{document}
\setcounter{tocdepth}{2}
\maketitle
\begin{abstract}
A minimal immersion from a surface to $S^3$ can be viewed both as a critical point of the area and of the energy. Although no difference appears at first order, looking at the respective second variations unveils significant differences. It is well known that whenever the first eigenvalue satisfies $\lambda_1(\Sigma)\geq2$, the index is $\mathrm{ind}_E(\Sigma)\leq 4$. The converse implication is much more subtle. We prove that whenever $\lambda_1(\Sigma)<\frac{1}{6}$, there exists a vector field $X$, orthogonal to the four Möbius vector fields, with negative second variation. We also prove an arbitrary codimension version of this statement: any immersed minimal surface $\Sigma\subset S^n$ with first eigenvalue $\lambda_1(\Sigma)<\frac{n-2}{2n}$ admits a vector field $X$ orthogonal to the $n+1$ Möbius fields with negative second variation.
{}\end{abstract}
\section*{Introduction}
A surface in $S^n$ for which the first variation of area vanishes is said to be a minimal surface and has zero mean curvature. The Morse index of a minimal surface is defined as the maximal dimension of a subspace of normal variations on which the second variation of area is negative definite. When looking at minimal surfaces in $S^3$, the index can be reduced to the scalar operator
\begin{equation}
    J_A(f)=\int_\Sigma |\nabla f|^2-2f^2-|A|^2f^2d\mu.
\end{equation}\\
Urbano \cite{Urbano} proved the classification of low index minimal surfaces in $S^3$: if the minimal surface has index at most five, then it is either a totally geodesic two-sphere, which has index $1$, or a Clifford torus, which has index $5$. For minimal surfaces in higher codimension, even in the case of spheres, the index is poorly understood. For codimension higher than one, the second variation operator cannot be reduced to an operator acting on functions, and many properties of the index change. Kusner--Wang recently studied the case of minimal $2$-tori in $S^4$ in \cite{KusnerPeng}. For ambient spaces with positive isotropic curvature, lower bounds for minimal spheres have been obtained by Micallef--Moore \cite{micallefmoore}. Many of the results regarding index of minimal surfaces in higher codimension rely on work of Ejiri--Micallef \cite{ejirimicallef}, relating the area index of a minimal surface, to the index of its immersion as a harmonic map. \\
In this note, we study the relation between the first eigenvalue of a minimal surface in $S^n$ and its energy index. The energy index is the maximal dimension of a subspace of vector fields tangent to the sphere, such that the operator
\begin{equation}
    D^2E(X)=\int_\Sigma |\nabla^{S^n}X|^2-2|X^N|^2-|X^T|^2d\mu
\end{equation}
is negative definite.
The operator can be rewritten as
\begin{equation}
    D^2E(X)=\sum\limits_i\int_\Sigma |\nabla X^i|^2-2|X^i|^2.
\end{equation}
For any surface $\Sigma$ which is not contained in an equatorial $S^2$, there is an explicit four-dimensional subspace on which $D^2E$ is negative definite, which is the space of Möbius vector fields
\begin{equation}
    \xi_i=\nabla^{S^n} x_i=e_i-\langle e_i,x\rangle x.
\end{equation}
Any vector field $X$ orthogonal to the Möbius fields satisfies
\begin{equation}
    \int_\Sigma X^i=\int_\Sigma X\cdot \xi_i=0. 
\end{equation}
Hence, by looking at the second variation of energy on the orthogonal complement of the Möbius vector fields, one observes that (see also \cite{karpukhin})
\begin{equation}
    \lambda_1(\Sigma)\geq 2\quad\text{implies }\quad D^2E(X)\geq 0, \text{ for all }X\perp \{\xi_1,...,\xi_{n+1}\}.
\end{equation}
Whether the converse holds is much more difficult to answer since the existence of a single function $f$ with 
\begin{equation}
    \int_\Sigma |\nabla f|^2-2f^2<0\quad \text{and}\quad\int_\Sigma f=0
\end{equation}
does not a priori imply the existence of a zero-average vector field tangent to $S^n$ with
\begin{equation}
    \int_\Sigma |\nabla X|^2-2|X|^2<0 \quad \text{and} \quad \int_\Sigma X=0.
\end{equation}
We prove in this note that one can construct such a vector field assuming $\lambda_1<\frac{n-2}{2n}$.
\\

Yau conjectured that for any embedded minimal surface $\Sigma\subset S^3$ the first eigenvalue is given by $\lambda_1(\Sigma)=2$. It would therefore be interesting to understand if\\ \vspace{0.3cm}

\textit{Question:} for a minimal immersed surface $\Sigma\subset S^3$, does $\lambda_1(\Sigma)<2$ imply that the energy index is at least $5$?\\
\vspace{0.5cm}

Our proof relies on identities which are independent of the codimension of the surface in $S^n$, which allow us to prove:
\begin{theorem}\label{maintheorem}
    Let $\Sigma\subset S^n$ be a minimal immersed surface with $\lambda_1(\Sigma)<\frac{n-2}{2n}$. Then,
    \begin{equation}
        \text{there exists }X\perp \xi_1,...,\xi_{n+1}\,\,\,\text{with }D^2E(X)<0. 
    \end{equation}
    Equivalently, there exists a vector field $X$ on the surface $\Sigma$, $\int_\Sigma X=0$ and $X(x)\cdot x\equiv 0$,
    \[
        \sum\limits_{i=1}^{n+1}\int_\Sigma |\nabla X^i|^2-2|X^i|^2<0.
    \]
\end{theorem}

\section{Preliminaries}
We start by introducing the notation used for minimal immersions, as well as the definitions of index for both the energy and area functionals.
Let $u:\Sigma\to S^n$ be an isometrically immersed minimal surface (closed, orientable). Then, (see Theorem 1.2 in \cite{brendlesurvey})
\begin{equation}
   - \Delta u= 2u=|\nabla u|^2 u.
\end{equation}
Given that $u$ is stationary both for the area and for the energy, one can consider the respective second variations. By definition,
\begin{equation}
    A(u)=\int_{\Sigma}\sqrt{|u_x|^2|u_y|^2-(u_x\cdot u_y)^2}dxdy\leq\frac{1}{2}\int_\Sigma|u_x|^2+|u_y|^2dxdy=E(u)
\end{equation}
which immediately implies for an arbitrary variation $u_t$ of $\Sigma$,
\begin{equation}\label{equationinequalityofsecondderivatives}
    \frac{d^2}{dt^2}\bigg\vert_{t=0}A(u_t)\leq  \frac{d^2}{dt^2}\bigg\vert_{t=0}E(u_t).
\end{equation}
\begin{definition}[index]
    Let $u_t$ be a normal variation of $\Sigma$, that is $\partial_tu_t\in N\Sigma$. In co-dimension one, for a smooth function $f\in C^\infty(\Sigma)$ and $\nu_\Sigma$ a choice of unit normal to $\Sigma$ in $S^3$, (see i.e. \cite{sharp})
\begin{equation}
    D^2A_u(f):=\frac{d^2}{dt^2}\bigg\vert_{t=0}A(u_t)=\int_\Sigma |\nabla f|^2-2f^2-|A_\Sigma|^2f^2
\end{equation}
where $|A_\Sigma|$ denotes the norm of the second fundamental form. In higher codimension the formula is more complicated and will not be needed here (see \cite{simons68}).The \textbf{area index} of $\Sigma$ is the maximal dimension of a subspace on which $D^2A$ is negative definite:
    \begin{equation*}
         \mathrm{ind}_A(u)=\mathrm{max}\big\{\mathrm{dim}(V)\vert V\subset W^{1,2}(\Sigma) \text{ is a linear subspace on which } D^2A_u<0\big\}.
    \end{equation*}
For an arbitrary (not necessarily normal) variation $\partial_tu_t=X\in \Gamma(TS^n\vert_\Sigma)$ of $u$ , the second variation of energy is given by (see \cite{takeuchinharm})
\begin{equation}
    D^2E_u(X):=\frac{d^2}{dt^2}\bigg\vert_{t=0}E(u_t)=\int_\Sigma |\nabla X|^2-\mathrm{R}_{S^n}(\partial_i,X,\partial_i,X)=\int_\Sigma |\nabla X|^2-2|X^\perp|^2-|X^T|^2
\end{equation}
for an orthonormal basis $\partial_1,\partial_2$ of $T\Sigma$. The \textbf{energy index} is defined to be the maximal dimension of a subspace on which $D^2E$ is negative definite:
 \begin{equation*}
        \mathrm{ind}_E(u)=\mathrm{max}\big\{\mathrm{dim}(V)\vert V\subset W^{1,2}(\Gamma(u^{-1}TS^n))\cap L^\infty \text{is a linear subspace with } D^2E_u<0\big\}
    \end{equation*}
\end{definition}

\vspace{0.5cm}

\subsection{Index comparison}
Given a minimal surface $\Sigma\subset S^n$, the two indices $\mathrm{ind}_A$ and $\mathrm{ind}_E$ are typically very different. For example, for the Lawson surfaces $\xi_{1,g}$, Kapouleas--Wyigul (\cite{kapouleaswiygul} proved $\mathrm{ind}_A(\xi_{1,g})=2g+3$, while $\mathrm{ind}_E(\xi_{1,g})=4$. As we explain in this note, it appears that the energy index is more closely related to the first eigenvalue of the minimal surface. The precise relation between the two indices, is given by the following:
\begin{theorem}[Ejiri--Micallef, \cite{ejirimicallef}] Let $u:\Sigma\to N$ be a branched minimal immersion. Then,
\begin{equation}
    \mathrm{ind}_E(u)\leq\mathrm{ind}_A(u)\leq\mathrm{ind}_E(u)+r
\end{equation}
for 
\begin{equation}
    r=\begin{cases}
 6g -6-2b\text{
if } b\leq 2g -3,\\
 4g -2+2\big[\frac{-b}
{2} \big]\text{ if } 2g-2\leq b\leq 4g-4,\\
 0 \text{
if }b \geq 4g -3.
    \end{cases}
\end{equation}
    
\end{theorem}

\section{Results}
Throughout this section, $\Sigma\to S^n$ is an isometric minimal immersion, with image not contained in a totally geodesic $S^2\subset S^n$.
\subsection{Computations with Möbius fields}
Consider the sections given by
\begin{equation}
    \xi_i (x)=e_i-\langle e_i,x\rangle x\in \mathbb{R}^{n+1},\quad \xi_i\perp x=0.
\end{equation}
Computing the covariant derivatives, for all $v\in T\Sigma$:
\begin{align}
    \nabla_v^{S^n}\xi_i&=-\nabla_v^{\mathbb{R}^{n+1}}(\langle e_i,x\rangle x)+\langle \nabla_v^{\mathbb{R}^{n+1}}(\langle e_i,x\rangle x), x\rangle x\\&= -\partial_v\langle e_i,x\rangle x-\langle e_i,x\rangle v + \partial_v\langle e_i,x\rangle x\\
    &= -\langle e_i,x\rangle v\\
    &=-x_iv.
\end{align}
Computing the second variation of energy in the direction $\xi_i$ (see \cite{Xin1}):
\begin{equation}
    D^2E_u(\xi_i)= \int_\Sigma|\nabla^{S^n}_u\xi_i|^2-2|\xi_i|^2+|\xi_i^T|^2=\int_\Sigma 2x_i^2-2|\xi_i|^2+|\xi_i^T|^2
\end{equation}
The second term is given by
\begin{equation}
    |\xi_i|^2=\langle e_i-\langle e_i,x\rangle x,e_i-\langle e_i,x\rangle x\rangle=1-2\langle e_i,x\rangle^2+\langle e_i,x\rangle^2=1-x_i^2,
\end{equation}
while the third term is (for a ONB $f_j$ of $T\Sigma$),
\begin{equation}
    |\xi_i^T|^2=\langle \xi_i,f_j\rangle^2=\langle e_i,f_j\rangle^2=|\nabla^\Sigma(e_i\cdot x)|^2=|\nabla^\Sigma x_i|^2.
\end{equation}
The second variation is therefore given by
\begin{equation}
    D^2E_\Sigma(\xi_i)=\int_\Sigma 2|\nabla^\Sigma x_i|^2-2|\xi_i|^2=\int_\Sigma 6x_i^2-2 = -2\int_\Sigma|\xi_i|^2 + 4\int_\Sigma |x_i|^2
\end{equation}
\begin{equation}\label{secondvariationmob:equation}
    =-2\int_{\Sigma}|\xi_i|^2+2|\xi_i^T|^2 = -2\int_{\Sigma}|\xi_i^N|^2.
\end{equation}
Since \eqref{secondvariationmob:equation} holds for any $\xi_v=\nabla^{S^3}\langle x,v\rangle$, $v\in\mathbb{R}^{n+1}$, we deduce:
\begin{lemma}[El Soufi \cite{elsoufi}]
    Let $u:\Sigma\to S^n$ be an isometric minimal immersion, such that the image of $u$ is not contained in a totally geodesic $S^2$. Then
    \begin{equation}
        \mathrm{ind}_E(u)\geq n+1.
    \end{equation}
\end{lemma}
\begin{remark}
    The Möbius fields $\xi_i$ are not eigensections for $D^2E$ in general. Their normal parts $\xi_i^\perp$ are eigensections for the second variation of area.
\end{remark}
\subsection{Canonical variations for functions}
\begin{proposition}[Canonical Möbius variations]\label{prop1}
    Let $u:\Sigma^2\to S^n$ be a minimal immersion and $f\in C^\infty(\Sigma)$. Then, for the ambient Moebius fields $\xi_i=\nabla^{S^n}x_i\big\vert_\Sigma$, 
    \begin{equation}
       \sum\limits_{i=1}^{n+1} D^2E(f\xi_i)= n\int_{\Sigma}|\nabla f|^2-(2n-4)\int_\Sigma |f|^2.
    \end{equation}
\end{proposition}
\begin{proof}
  \begin{align}
        D^2E(f\xi_i)&=\int_\Sigma |\nabla f|^2|\xi_i|^2+f^2|\nabla \xi_i|^2+2f\nabla f\xi_i\nabla \xi_i-2f^2|\xi_i|^2+f^2|\xi_i^T|^2\\
        &=\int_\Sigma |\nabla f|^2|\xi_i|^2+f^2|\nabla \xi_i|^2+f\nabla f\nabla |\xi_i|^2-2f^2|\xi_i|^2+f^2|\xi_i^T|^2\\
       &=\int_\Sigma |\nabla f|^2|\xi_i|^2+f^2|\nabla \xi_i|^2-\frac{1}{2}f^2\Delta |\xi_i|^2-2f^2|\xi_i|^2+f^2|\xi_i^T|^2\\
       &=\int_\Sigma |\nabla f|^2(1-x_i^2)+f^22x_i^2-\frac{1}{2}f^2\Delta |\xi_i|^2-2f^2(1-x_i^2)+f^2|\xi_i^T|^2
    \end{align}
    Therefore, summing over $i$,
    \begin{equation}
        \sum\limits_{i=1}^{n+1}D^2E(f\xi_i)=n\int_\Sigma |\nabla f|^2-\int_\Sigma(2n-4)f^2
    \end{equation}
\end{proof}

Next, we prove that if $\Sigma$ has small first eigenvalue $\lambda_1$, the corresponding Möbius variations $f\xi_i$ can be used to obtain a negative direction for $D^2E$ which is orthogonal to all Möbius fields.
\vspace{-0.3cm}
\begin{proposition}\label{proposition}
     Let $u:\Sigma^2\to S^n$ be a minimal immersion and $f\in C^\infty(\Sigma)$ a solution of $-\Delta f=\lambda f$ with $\lambda\leq 1$. Then, for the ambient Möbius field $\xi_i=\nabla^{S^n}x_i\big\vert_\Sigma$, if 
    \begin{equation}
     D^2E(f\xi_i)= <-\frac{3}{2}\int_\Sigma |f\xi_i^N|^2,
    \end{equation}
    the projection $X_f^i$ orthogonal to all Möbius fields $\xi_v$ satisfies
    \begin{equation}
        D^2E(X_f^i)<0.
    \end{equation}
\end{proposition}

\begin{proof}
     Since $\int f=0$, $-\Delta x_i=2x_i$ and $\xi_i\cdot \xi_j=\delta_{ij}-x_ix_j$,

\[
\int_\Sigma -f\xi_ia_j\xi_j=\int_\Sigma f a_j x_i x_j
= -\frac{1}{\lambda}\int_\Sigma\Delta f a_jx_ix_j =-\frac{2}{\lambda} \int_\Sigma f \nabla x_i \cdot a_j \nabla x_j 
+ \frac{4 }{\lambda} \int_\Sigma f x_i a_jx_j
\]

\noindent Rearranging and using $\nabla^\Sigma x_i=\xi_i^T$,

\[
\left( \lambda - 4\right) \int_\Sigma f a_j x_i x_j =(4-\lambda)\int_\Sigma f\xi_i\cdot a_j\xi_j
= - 2 \int_\Sigma f \nabla x_i \cdot a_j\nabla x_j = -2\int_\Sigma f\xi_i^T\cdot a_j\xi_j^T
\]

\noindent hence:
\begin{equation}\label{55}
    \int_\Sigma f\xi_i\cdot a_j\xi_j=-\frac{2}{4-\lambda}\int_\Sigma f\xi_i^T\cdot a_j\xi_j^T
\end{equation}
\begin{equation}\label{3}
    \int f\xi^N_ia_j\xi_j^N=-\frac{6-\lambda}{4-\lambda}\int f\xi^T_ia_j\xi_j^T=\frac{6-\lambda}{2}\int f\xi_i\cdot a_j\xi_j
\end{equation}

 \noindent Consider now an $L^2$ orthogonal decomposition $f\xi_1=X\oplus a_j\xi_j$.
    \begin{align}
        D^2E(X)
        &= D^2E(f\xi_1)+D^2E(a_j\xi_j)-2\int \nabla(f\xi_1)\nabla(a_j\xi_j)+4\int f\xi_1^Na_j\xi_j^N+2\int f\xi_1^T a_j\xi_j^T\label{1} 
    \end{align}
    The mixed gradient term satisfies
    \begin{equation}\label{2}
      -2  \int \nabla(f\xi_1)\nabla (a_j\xi_j)=2\int f\xi_1(\nabla_v( -a_jx_jv))=-2\int f\xi_1^Ta_j\xi_j^T
    \end{equation}
    By \eqref{1} and \eqref{2},
    \begin{align}
       D^2E(X) =&D^2E(f\xi_1)+D^2E(a_j\xi_j)-2\int \nabla(f\xi_1)\nabla(a_j\xi_j)+4\int f\xi_1a_j\xi_j-2\int f\xi_1^T a_j\xi_j^T\\
        =& D^2E(f\xi_1)+D^2E(a_j\xi_j) +4\int f\xi_i^Na_j\xi_j^N
    \end{align}
  
    \noindent and hence
    \begin{equation}
        D^2E(X)=D^2E(f\xi_1)-2\int |a_j\xi_j^N|^2+4\int f\xi_1^Na_j\xi_j^N.
    \end{equation}
    Using $\int Xa_j\xi_j=0$, and the equations satisfied by $f$ and $\xi_v$ we compute:
    \begin{align}
        \int |f\xi_1^N|^2=&\int |a_j\xi_j^N|^2+\int |X^N|^2 +2\int X^Na_j\xi_j^N\\
        =&\int |a_j\xi_j^N|^2+\int |X^N|^2-2\int X^Ta_j\xi_j^T\\=& \int |a_j\xi_j^N|^2+\int |X^N|^2-2\int f\xi_1^Ta_j\xi_j^T+2\int |a_j\xi_j^T|^2\\
        \overset{\eqref{55}}{=} &\int |a_j\xi_j^N|^2+\int |X^N|^2+{(4-\lambda)}\int f\xi_1a_j\xi_j+2\int |a_j\xi_j^T|^2\\
        =&\int |a_j\xi_j^N|^2+\int |X^N|^2+{(4-\lambda)}\int |a_j\xi_j|^2+2\int |a_j\xi_j^T|^2\\=&(5-\lambda)\int |a_j\xi_j^N|^2+\int |X^N|^2+{(6-\lambda)}\int |a_j\xi_j^T|^2\label{44}\\
    \end{align}
    as well as the tangential contribution
      \begin{align}
        \int |f\xi_1^T|^2=&\int |a_j\xi_j^T|^2+\int |X^T|^2 +2\int X^Ta_j\xi_j^T\\
        =&\int |a_j\xi_j^T|^2+\int |X^T|^2-2\int X^Na_j\xi_j^N\\=& \int |a_j\xi_j^T|^2+\int |X^T|^2-2\int f\xi_1^Na_j\xi_j^N+2\int |a_j\xi_j^N|^2\\
        \overset{\eqref{55}}{=} &\int |a_j\xi_j^T|^2+\int |X^T|^2-{(6-\lambda)}\int f\xi_1a_j\xi_j+2\int |a_j\xi_j^N|^2\\
        =&\int |a_j\xi_j^T|^2+\int |X^T|^2-{(6-\lambda)}\int |a_j\xi_j|^2+2\int |a_j\xi_j^N|^2\\=&-(5-\lambda)\int |a_j\xi_j^T|^2+\int |X^T|^2-{(4-\lambda)}\int |a_j\xi_j^N|^2
    \end{align}
    Combining the two, we obtain: 
    \begin{align}
       & 2\int X^Ta_j\xi_j^T =-2\int X^Na_j\xi_j^N=-(6-\lambda)\int |a_j\xi_j^T|^2-(4-\lambda)\int |a_j\xi_j^N|^2\\
       \implies & -2\int f\xi_1^Na_j\xi_j^N + 2\int |a_j\xi_j^N|^2 = -(6-\lambda)\int |a_j\xi_j^T|^2-(4-\lambda)\int |a_j\xi_j^N|^2\\
       \implies &4\int f\xi_1^Na_j\xi_j^N= 2(6-\lambda)\int |a_j\xi_j^T|^2+2(6-\lambda)\int |a_j\xi_j^N|^2
    \end{align}
    By Hölder's inequality
    \begin{equation}\label{888}
        2(6-\lambda)\int |a_j\xi_j^T|^2+2(6-\lambda)\int |a_j\xi_j^N|^2\leq \frac{2}{\gamma}\int |f\xi_1^N|^2+2\gamma\int |a_j\xi_j^N|^2.
    \end{equation}
    By \eqref{44} and \eqref{888} with the choice $\gamma=2$,
    \begin{align}
        &\frac{3}{2}\int |f\xi_1^N|^2=\int |f\xi_1^N|^2+\frac{1}{2}\int |f\xi_1^N|^2\\\geq &\, 2(6-\lambda)\int |a_j\xi_j^T|^2+2(4-\lambda)\int |a_j\xi_j^N|^2+\frac{1}{2}\big((5-\lambda)\int |a_j\xi_j^N|^2+\int |X^N|^2+{(6-\lambda)}\int |a_j\xi_j^T|^2\big)\\
       \geq &\,\frac{5}{2}(6-\lambda)\int |a_j\xi_j^T|^2+\big(2(4-\lambda)+\frac{5-\lambda}{2}\big)\int |a_j\xi_j^N|^2.
    \end{align}
    Hence, 
    \begin{align}
        &\quad \frac{3}{2}\int |f\xi_1^N|^2+2\int |a_j\xi_j^N|^2
        \\
       &\geq \frac{5}{2}(6-\lambda)\int |a_j\xi_j^T|^2+\big(2(4-\lambda)+\frac{5-\lambda}{2}+2\big)\int |a_j\xi_j^N|^2\\
       &\geq 2(6-\lambda)\int |a_j\xi_j^T|^2+2(6-\lambda)\int |a_j\xi_j^N|^2\quad\text{for }\lambda<1\\
       &= 4\int f\xi_1^Na_j\xi_j^N.
    \end{align}
   where we used that $\frac{5-\lambda}{2}\geq 2$ whenever $\lambda\leq 1$. Finally, we compute
    \begin{align}
        D^2E(X)&=D^2E(a_j\xi_j)+D^2E(f\xi_1)+4\int f\xi_1^Na_j\xi_j^N\\&< 
        -2\int |a_j\xi_j^N|^2-\frac{3}{2}\int |f\xi_1^N|^2+4\int f\xi_1^Na_j\xi_j^N\\&<0.
    \end{align}
which concludes the proof of the Proposition.
\end{proof}
We can now prove the main Theorem:
\begin{theorem}
    Let $\Sigma\subset S^n$ be a minimal immersed surface with $\lambda_1(\Sigma)<\frac{n-2}{2n}$. Then,
    \begin{equation}
        \text{there exists }X\perp \xi_1,...,\xi_{n+1}\,\,\,\text{with }D^2E(X)<0. 
    \end{equation}
    Equivalently, there exists a vector field $X$ on the surface $\Sigma$, $\int_\Sigma X=0$ and $X(x)\cdot x\equiv 0$,
    \[
        \sum\limits_{i=1}^{n+1}\int_\Sigma |\nabla X^i|^2-2|X^i|^2<0.
    \]
\end{theorem}
\begin{proof}
    By Proposition \ref{proposition} it suffices to prove that there exists $f\xi_i$ satisfying
    \begin{equation}
        D^2E(f\xi_i)<-\frac{3}{2}\int_\Sigma |f\xi_i^N|^2.
    \end{equation}
    By Proposition \ref{prop1}, applied to $-\Delta f=\lambda_1 f$,
     \begin{equation}
       \sum\limits_{i=1}^{n+1} D^2E(f\xi_i)= n\int_{\Sigma}|\nabla f|^2-(2n-4)\int_\Sigma f^2=(n\lambda-2n+4)\int_\Sigma f^2.
    \end{equation}
    Using that $\sum\limits_{i=1}^{n+1}|\xi_i^\perp|^2=n-2$,
     \begin{equation}
       \sum\limits_{i=1}^{n+1} D^2E(f\xi_i)=\sum\limits_{i=1}^{n+1}\frac{n\lambda-2n+4}{n-2}\int_\Sigma |f\xi_i^N|^2.
    \end{equation}
   In particular, there exists $i_0\in\{1,...,n+1\}$ with
   \begin{equation}
       D^2E(f\xi_i)< \frac{n\lambda-2n+4}{n-2}\int_\Sigma |f\xi_i^N|^2
   \end{equation}
   and it suffices to verify
   \begin{equation}
       \frac{n\lambda-2n+4}{n-2}<-\frac{3}{2}\iff 2n\lambda-4n+8<-3n+6\iff 2n\lambda+2<n\iff \lambda <\frac{n-2}{2n}.
   \end{equation}
   This concludes the proof of the main Theorem.
\end{proof}

\end{document}